\documentclass{amsart}
\usepackage{amssymb}
\usepackage{amscd}
\usepackage{graphicx}
\usepackage{amsmath}
\usepackage{setspace}
\usepackage[all]{xy}
\usepackage{color}
\usepackage{enumerate}

\newtheorem{theorem}{Theorem}[section]

\newtheorem{corollary}[theorem]{Corollary}

\newtheorem{lemma}[theorem]{Lemma}

\newtheorem{proposition}[theorem]{Proposition}

\theoremstyle{remark}
\newtheorem{definition}[theorem]{Definition}
\newtheorem{example}[theorem]{Example}
\newtheorem{remark}[theorem]{Remark}

\newcommand{\CB}{\mathcal{B}}

\newcommand{\CALD}{\mathcal{D}}

\newcommand{\CI}{\mathcal{I}}

\newcommand{\CT}{\mathcal{T}}

\newcommand{\BBQ}{{\mathbb Q}}

\newcommand{\BBZ}{{\mathbb Z}}
\newcommand{\Z}{{\mathbb Z}}
\newcommand{\PP}{{\mathbb P}}

\newcommand{\Hom}{\mathrm{Hom}}
\newcommand{\Ima}{\mathrm{Im}}

\newcommand{\Ker}{\mathrm{Ker}}

\newcommand{\Ext}{\mathrm{Ext}}
\newcommand{\Tor}{\mathrm{Tor}}

\newcommand{\Gen}{\mathrm{Gen}}

\newcommand{\Cogen}{\mathrm{Cogen}}

\newcommand{\Copres}{\mathrm{Copres}}

\newcommand{\Add}{\mathrm{ Add}}

\newcommand{\Prod}{\mathrm{ Prod}}

\newcommand{\Modr}{\mathrm{ Mod}\text{-}}
\newcommand{\Mod}{\text{-}\mathrm{ Mod}}

\begin{document}


\title{Cosilting Modules}

\author{Simion Breaz and Flaviu Pop}
\thanks{S. Breaz is supported by the Babe\c s-Bolyai University grant GSCE-30254}

\address{Simion Breaz: "Babe\c s-Bolyai" University, Faculty of Mathematics and Computer Science, Str. Mihail Kog\u alniceanu 1, 400084, Cluj-Napoca, Romania}

\email{bodo@math.ubbcluj.ro}

\address{Flaviu Pop: "Babe\c s-Bolyai" University, Faculty of Economics and Business Administration, str. T. Mihali, nr. 58-60, 400591, Cluj-Napoca, Romania}

\email{flaviu.v@gmail.com; flaviu.pop@econ.ubbcluj.ro}




\begin{abstract}
We study the class of modules, called cosilting modules, which are defined as the categorical duals of silting modules. Several characterizations of
these modules and connections with silting modules are presented. We prove that Bazzoni's theorem about the pure-injectivity of cotilting modules is also valid for cosilting modules.
\end{abstract}

\maketitle

\section{Introduction}
Silting objects in triangulated categories, introduced by Keller and Vossieck in \cite{Ke_Vo:1998} for bounded derived categories, are important tools in the study of homotopy or derived categories since they are in  correspondence with other important concepts such as (co-)$t$-structures or simple-minded collections of objects (see \cite{Ko_Ya:2014}).
In \cite{Angeleri_Marks_Vitoria:2015} the authors introduced the notion of (partial) silting $R$-modules (where $R$ is a unital associative ring) in order to study, in the category of all right $R$-modules, the class of cokernels of those homomorphisms between projective modules which represent, as $2$-term complexes, silting objects in the derived category $\mathbf{D}(R)$.

In this paper we study the modules which are defined as categorical duals (in the category of all right $R$-modules) of silting modules.
Therefore  (partial) cosilting modules are kernels of homomorphisms $\zeta$ between
injective modules such that the class $B_\zeta$, consisting of all modules $X$ with the
property that $\Hom_R(X,\zeta)$ is an epimorphism, has some closure properties which are
similar to those properties satisfied by the Ext-orthogonal classes associated to
(partial) cotilting modules. In fact, the class of all (partial) cosilting modules
contains all (partial) cotilting modules.
 The main characterization of cosilting modules presented in this paper is Theorem \ref{main-cosilting}, where cosilting modules are described as those partial cosilting modules such that $\CB_\zeta$ has a special precovering property. Using this theorem  we prove that every partial cosilting module is a direct summand of a cosilting module (Theorem \ref{pcs-ds}). We also describe some connections of (partial) cosilting modules with (partial) silting modules (Proposition \ref{dual-partial} and Corollary \ref{dual}).  Then we show in Theorem \ref{cosi-pi} that all (partial) cosilting modules are pure-injective via the techniques used by Bazzoni \cite{Bazzoni:2003} (where she proved that all cotilting modules are pure-injective).

Throughout this paper, by a ring $R$ we will understand a unital associative ring, an $R$-module is a right $R$-module and we will denote by $\Modr R$ (respectively, by $R \Mod$) the category of all right (respectively, left) $R$-modules. For an $R$-module $T$, we consider the perpendicular class ${^{\perp}T}$ defined as follows $${^{\perp}T} = \{ X \in \Modr R \mid \Ext^{1}_{R}(X,T) = 0 \}.$$
An $R$-module is said to be $T$-generated if it is an epimorphic image of a direct sum of copies of $T$ and an $R$-module is called $T$-cogenerated if it can be embedded into a direct product of copies of $T$. By $\Gen(T)$ (respectively, by $\Cogen(T)$) we will denote the class of all $T$-generated (respectively, $T$-cogenerated) $R$-modules. We also denote by $\Add(T)$ (respectively, by $\Prod(T)$) the class of all modules which are isomorphic to direct summands of direct sums (respectively, of direct products) of copies of $T$.  We refer to \cite{Gobel_Trlifaj:2006} for other notions and notations used in this paper.

\section{The class $\CB_\zeta$}

In \cite{Angeleri_Marks_Vitoria:2015}, the authors introduce the notion of {\sl silting module} by defining the class $$\CALD_{\sigma} = \{ X\in\Modr R \mid \Hom_{R}(\sigma,X) \text{ is an epimorphism}\}$$where $\sigma:P_{1} \to P_{0}$ is an $R$-homomorphism of (projective) $R$-modules. In order to define the dual notion, of {\sl cosilting module}, we define the class $\CB_{\zeta}$ as follows. If $\zeta: Q_{0} \to Q_{1}$ is an $R$-homomorphism, then the class $\CB_{\zeta}$ is defined as
$$\CB_{\zeta}=\{X \in \Modr R \mid \Hom_{R}(X,\zeta) \text{ is an epimorphism} \}.$$

We mention that the results presented in this section, also work in the dual case.

Dualizing the notion of defect functor considered in \cite{Br_Ze:2014} and \cite{Kr:2003}, for an $R$-homomorphism $\zeta:Q_0\to Q_1$ we can consider the codefect functor which assigns to every module $X$ the abelian group $\mathrm{Coker}\Hom_R(X,\zeta)$. In this context the class $\CB_\zeta$
is in fact the kernel of this functor.

The proof of the following lemma is a simple exercise.

\begin{lemma}\label{sigma_decomposition} Let $\zeta: Q_{0} \to Q_{1}$ be an $R$-homomorphism and assume that $\zeta = \tau_{\zeta} \circ \pi_{\zeta}$ is the canonical decomposition of $\zeta$. The following statements are equivalent for an $R$-module $X$:
\begin{enumerate}[{\rm (1)}]
\item $X\in \CB_{\zeta}$;

\item $\Hom_{R}(X,\tau_{\zeta})$ is an isomorphism and $\Hom_{R}(X,\pi_{\zeta})$ is an epimorphism.
\end{enumerate}
\end{lemma}



\begin{corollary}\label{remark_isomorphism} Let $\zeta: Q_{0} \to Q_{1}$ be an $R$-homomorphism with $T=\Ker(\zeta)$ and let $\zeta = \tau_{\zeta} \circ \pi_{\zeta}$ be the canonical decomposition. The following statements are equivalent for an $R$-module $X$ which belongs to ${^{\perp}T}$:
\begin{enumerate}[{\rm (1)}]

\item $X\in\CB_{\zeta}$;

\item $\Hom_{R}(X,\tau_{\zeta})$ is an isomorphism.
\end{enumerate}
\end{corollary}


In the following results we establish some closure properties of the class $\CB_{\zeta}$.


\begin{lemma}\label{closure_prop} \label{last_term} Let $\zeta: Q_{0} \to Q_{1}$ be an $R$-homomorphism.
   \begin{enumerate}[{\rm (1)}]

		\item The class $\CB_{\zeta}$ is closed under direct sums.

					\item If $Q_{1}$ is injective then the class $\CB_{\zeta}$ is closed under submodules.

					\item If $Q_{0}$ is injective then the class $\CB_{\zeta}$ is closed under extensions.

						\item If  $Q_{0}$ is injective and $T=\Ker(\zeta)$ then $\CB_{\zeta} \subseteq {^{\perp}T}$.

\item	Assume that	$Q_{1}$ is injective. If  $T=\Ker(\zeta)$ and  $$0 \to A \overset{f}\longrightarrow B \overset{g}\longrightarrow X \to 0$$ is an exact sequence such that $A$ and $B$ belong to the class $\CB_{\zeta}$ and $X \in {^{\perp}T}$ then $X\in\CB_{\zeta}$.
   \end{enumerate}
\end{lemma}
\begin{proof}(1) This follows from the fact that every contravariant $\Hom$ functor sends direct sums into direct products.

(2) Let $Y$ be in $\CB_{\zeta}$ and let $f:X\to Y$ be a monomorphism. If $h \in \Hom_{R}(X,Q_{1})$ then, by the injectivity of $Q_{1}$, there is an homomorphism $u:Y \to Q_{1}$ such that $h = uf$. Since $\Hom_R(Y,\zeta)$ is an epimorphism, there exists $v \in \Hom_{R}(Y,Q_{0})$ with $u = \zeta v$, so we have the following commutative diagram
\[
\xymatrix{
&Q_{1} &Q_{0}\ar[l]_{\zeta}\\
0\ar[r] &X\ar[r]^{f}\ar[u]^{h} &Y\ar[ul]_{u}\ar[u]_{v}\\
}
\]
such that the bottom row is exact.

Then $\Hom_{R}(X,\zeta)(vf) = h$. It follows that $\Hom_{R}(X,\zeta)$ is an epimorphism, hence $X\in\CB_{\zeta}$.

(3) Let $0 \to X \overset{f}\longrightarrow Y \overset{g}\longrightarrow Z \to 0$ be an exact sequence with $X,Z\in\CB_{\zeta}$.
If $h\in\Hom_{R}(Y,Q_{1})$ there is $k\in\Hom_{R}(X,Q_{0})$ such that $\zeta k = hf$ (since $\Hom_{R}(X,\zeta)$ is epic).
Since $Q_{0}$ is injective, there is $u:Y \to Q_{0}$ with $k=uf$. Because $(\zeta u -h)f = 0$, there exists $v:Z \to Q_{1}$ such that $\zeta u-h = vg$. But $Z\in\CB_{\zeta}$, hence there is $r \in \Hom_{R}(Z,Q_{0})$ with $\zeta r = v$.
All these homomorphisms are represented in the following diagram:
\[
\xymatrix{
0\ar[r] &X\ar[r]^{f}\ar[dr]^{k} &Y\ar[r]^{g}\ar[d]_{u}\ar[dr]^{\ \ r} &Z\ar[r]\ar@{-->}[d]^{v}\ar@{-->}[dl]_{\ \ \ \ \ h} &0\\
& &Q_{0}\ar[r]_{\zeta} &Q_{1}.\\
}
\]
Therefore $\Hom_{R}(Y,\zeta)(u-rg)=h$, and we obtain that $\Hom_{R}(Y,\zeta)$ is an epimorphism since $h$ was arbitrarily. Hence
$Y\in\CB_{\zeta}$, and the proof is complete.

(4) Let $X \in \CB_{\zeta}$. Applying the covariant $\Hom_{R}(X,-)$ functor to the short exact sequence
$$0 \to T \overset{i}\longrightarrow Q_{0} \overset{\pi_{\zeta}}\longrightarrow \Ima(\zeta) \to 0,$$ where $\zeta = \tau_{\zeta} \circ \pi_{\zeta}$ is the canonical decomposition, we obtain the exact sequence
$$0 \to \Hom_{R}(X,T) \longrightarrow \Hom_{R}(X,Q_{0}) \longrightarrow \Hom_{R}(X,\Ima(\zeta)) \longrightarrow \Ext^{1}_{R}(X,T) \to 0.$$ By Lemma \ref{sigma_decomposition}, we have that the homomorphism $\Hom_{R}(X,\pi_{\zeta})$ is an epimorphism, hence $\Ext^{1}_{R}(X,T) = 0$. It follows that $X \in {^{\perp}T}$.

(5) Suppose that $\zeta = \tau_{\zeta} \circ \pi_{\zeta}$ is the canonical decomposition. Applying the contravariant functor
$\Hom_{R}(-,\Ima(\zeta))$, we have the exact sequence
$$0 \to \Hom_{R}(X,\Ima(\zeta)) \longrightarrow \Hom_{R}(B,\Ima(\zeta)) \longrightarrow \Hom_{R}(A,\Ima(\zeta)).$$
We also apply the contravariant functor $\Hom_{R}(-,Q_{1})$, and 
we obtain the exact sequence
$$0 \to \Hom_{R}(X,Q_{1}) \longrightarrow \Hom_{R}(B,Q_{1}) \longrightarrow \Hom_{R}(A,Q_{1}) \to 0.$$
Since $A$ and $B$ lie in $\CB_{\zeta}$, it follows by Lemma \ref{sigma_decomposition}, that both $\Hom_{R}(A,\tau_{\zeta})$ and $\Hom_{R}(B,\tau_{\zeta})$ are isomorphisms. It follows that $\Hom_{R}(f,\Ima({\zeta}))$ is an epimorphism, hence we have the following commutative diagram with exact rows
\[\xymatrix{0 \ar[r] & \Hom_{R}(X,\Ima(\zeta)) \ar[r] \ar[d] &
\Hom_{R}(B,\Ima(\zeta)) \ar[r] \ar[d] & \Hom_{R}(A,\Ima(\zeta)) \ar[r] \ar[d] & 0\\
0 \ar[r] & \Hom_{R}(X,Q_{1}) \ar[r] & \Hom_{R}(B,Q_{1}) \ar[r] & \Hom_{R}(A,Q_{1}) \ar[r] & 0
}\]


Applying the Snake Lemma we get that $\Hom_R(X,\tau_\zeta)$ is an isomorphism. The conclusion follows from Corollary \ref{remark_isomorphism}.
\end{proof}

\section{Cosilting Modules}


\begin{definition}
We say that an $R$-module $T$ is:
   \begin{enumerate}[{(1)}]

		\item {\sl partial cosilting} ({\sl with respect to $\zeta$}), if there exists an injective copresentation of $T$ $$0 \to T \overset{f}\longrightarrow Q_{0} \overset{\zeta}\longrightarrow Q_{1}$$such that:
        \begin{enumerate}[(a)]
           \item $T\in\CB_{\zeta}$, and
           \item the class $\CB_{\zeta}$ is closed under direct products;
        \end{enumerate}

			\item {\sl cosilting} ({\sl with respect to $\zeta$}), if there exists an injective copresentation $$0 \to T \overset{f}\longrightarrow Q_{0} \overset{\zeta}\longrightarrow Q_{1}$$ of $T$ such that $\Cogen(T)=\CB_{\zeta}$.
   \end{enumerate}
\end{definition}

\begin{remark}
If $\zeta:Q_0\to Q_1$ is a homomorphism such that $Q_{0}$ and $Q_{1}$ are injective $R$-modules
then the class $\CB_{\zeta}$ is closed under submodules and extensions (Lemma \ref{closure_prop}).
Hence the condition (b) in the definition of partial cosilting module is equivalent to the fact that the class $\CB_\zeta$ is a torsion-free class. Moreover, since $\Cogen(T)$ is closed under direct products, it is easy to see that every cosilting module is partial cosilting.
\end{remark}



In the following example we will see that every (partial) cotilting module is (partial) cosilting and, for every ring $R$, the trivial module $0$ is cosilting. Moreover, even for some hereditary
rings there exists non-zero cosilting modules which are not cotilting.

\begin{example}\label{si.vs.ti} (a) Let us recall, from \cite{Colpi_Tonolo_Trlifaj:1997}, that a right $R$-module $T$ is \textsl{partial cotilting} if and only if $\Cogen(T) \subseteq {^{\perp}T}$ and the class ${^{\perp}T}$ is a torsion-free class. Moreover, $T$ is \textsl{cotilting} if  and only if $\Cogen(T) = {^{\perp}T}$.
We note that the class ${^{\perp}T}$ is closed under submodules if and only if $\mathrm{id}(T) \leq 1$ ($\mathrm{id}(T)$ denotes the injective
dimension of $T$), so every partial cotilting module is of injective dimension at most 1.

Let $T$ be an $R$-module of injective dimension at most 1, and consider an exact sequence $$0 \to T \overset{f}\longrightarrow Q_{0} \overset{\zeta}\longrightarrow Q_{1} \to 0$$ with $Q_0$ and $Q_1$ injective $R$-modules. \textit{Then $T$ is (partial) cotilting if and only if $T$ is (partial) cosilting with respect to $\zeta$.}

The direct implications are obvious since $\CB_\zeta={^{\perp}T}$.
If we assume that $T$ is partial cosilting with respect to $\zeta$, we have
$\Cogen(T) \subseteq \CB_\zeta={^{\perp}T}$. Since $\CB_\zeta$ is closed under direct products it follows that $T$ is partial cotilting module. Moreover, if $T$ is cosilting with respect to $\zeta$, we obtain $\Cogen(T) = \CB_{\zeta} = {^{\perp}T}$, hence $T$ is cotilting.

(b) For every ring $R$ we consider an injective cogenerator $E$ for $\Modr R$. If $\zeta:0\to E$ is the trivial homomorphism then $\CB_\zeta$ contains only the trivial module, hence $0=\Ker(\zeta)$ is a cosilting module.

(c) Let $\kappa$ be a field, and let $R$ be the ring of all upper $2\times 2$ triangular matrices over $\kappa$. In the category
$\Modr R$ there are two simple modules $S=(0,\kappa)$ and $T=Q/S$, where $Q=(\kappa,\kappa)$. If $\zeta:Q\to T\oplus T$ is a homomorphism
such that $\Ker(\zeta)=S$, then  $X\in\CB_\zeta$ if and only if $\Ext^1_R(X,S)=0$ and $\Hom_R(X,T)=0$.

{Let $X\in \CB_\zeta$. Since $Q\oplus T$ is an injective cogenerator for $\Modr R$, there exists a set $I$ such that
$X$ can be embedded in $Q^I\oplus T^I$. Moreover, we have $\Hom_R(X,T)=0$, and it follows that $X$ can be embedded in $Q^I$. Therefore, we can suppose that $X$ is a submodule of the direct product $Q^I$. Since the direct products are exact in module categories, we can view $S^I$ as a submodule of $Q^I$ such that $Q^I/S^I\cong T^I$. Therefore, $\frac{X}{X\cap S^I}\cong \frac{X+S^I}{S^I}$ can be embedded in $T^I$. Using again $\Hom_R(X,T)=0$, it follows that $X\subseteq S^I$, hence $X\in\Cogen(S)$. Then $\CB_\zeta\subseteq \Cogen(S)$.

Conversely, if $X\in \Cogen(S)$, it is easy to see that there exists a set $I$ such that $X$ is isomorphic to  $S^{(I)}$. It follows that $\Ext^1_R(X,S)=0$ and $\Hom_R(X,T)=0$, hence $X\in\CB_\zeta$

Therefore
$\CB_\zeta=\Cogen(S)$, hence $S$ is cosilting. On the other hand, $S$ is not a cotilting module, since $Q\in {^{\perp}S}\setminus \Cogen(S)$.
}
\end{example}


Applying Lemma \ref{closure_prop} we obtain:

\begin{lemma}\label{inclusions} If the $R$-module $T$ is partial cosilting with respect to the injective copresentation $\zeta : Q_{0} \to Q_{1}$, then $\Cogen(T) \subseteq \CB_{\zeta} \subseteq {^{\perp}T}.$
\end{lemma}


Some of the properties of (partial) cotilting modules are still valid for (partial) cosilting modules. We recall that the class $\Copres(T)$ is the class
of all modules $X$ which are kernels of homomorphisms between direct products of copies of $T$. In the following we will denote by
${^{\circ}T} = \{ X \in \Modr R \mid \Hom_{R}(X,T)=0\},$
the torsion class induced by $T$.

\begin{corollary}\label{torsion_pair}\label{cogen=copres} Let $\zeta:Q_0\to Q_1$ be a homomorphism between injective modules. If $T=\Ker(\zeta)$, then the following are true:
\begin{enumerate}[{\rm (1)}]
\item If $T$ is partial cosilting then the pair $({^{\circ}T},\Cogen(T))$ is a torsion pair.

\item $T$ is cosilting with respect to $\zeta$ if and only if  $({^{\circ}T},\CB_\zeta)$ is a torsion pair.

\item If $T$ is a cosilting $R$-module then $\Cogen(T) = \Copres(T)$.
\end{enumerate}
\end{corollary}
\begin{proof} (1) This follows from the inclusion $\Cogen(T)\subseteq {^{\perp}T}$ and \cite[Proposition 1.4.2]{Colby_Fuller:2004} by a standard proof. For reader's convenience, we present the details of this proof.

Suppose that $T$ is partial cosilting $R$-module with respect to the injective copresentation $\zeta : Q_{0} \to Q_{1}$. We will prove that the class $\Cogen(T)$ is a torsion-free class. Since $\Cogen(T)$ is closed under both submodules and direct products, it is enough to prove that $\Cogen(T)$ is closed under extensions. If $S$ is the endomorphism ring of $T$, let us denote by $\Delta$ both contravariant functors $\Hom_{R}(-,T)$ and $\Hom_{S}(-,T)$. We recall that a module $X$ is $T$-cogenerated if and only if the canonical homomorphism $\delta_{X}:X\to\Delta^{2}(X)$ is a monomorphism.

Let $0 \to X \to Y \to Z \to 0$ be an exact sequence such that $X,Z\in\Cogen(T)$. Since $\Cogen(T)\subseteq {^{\perp}T}$, we have that $\Ext^{1}_{R}(Z,T)=0$, hence the induced sequence $$0 \to \Delta(Z) \to \Delta(Y) \to \Delta(X) \to 0$$ is also exact. Applying Snake Lemma to the following commutative diagram with exact rows

$$\begin{CD} 0@>>> X@>>> Y@>>> Z@>>> 0\\
 @. @VV{\delta_{X}}V @VV{\delta_{Y}}V @VV{\delta_{Z}}V @.\\
0@>>> \Delta^{2}(X)@>>> \Delta^{2}(Y)@>>> \Delta^{2}(Z)@. \ ,
\end{CD}$$
we obtain that $\delta_{Y}$ is a monomorphism, hence $Y\in\Cogen(T)$. Thus the class $\Cogen(T)$ is closed under extensions.

By \cite[Proposition 1.4.2]{Colby_Fuller:2004}, we have that $(\CT,\Cogen(T))$ is a torsion pair, where $$\CT=\{ X \in \Modr R \mid \Hom_{R}(X,F) = 0, \text{ for all } F \in \Cogen(T) \}.$$ It is easy to show that $\CT = {^{\circ}T}$ and the proof is complete.

(2) The direct implication is a consequence of (1). Conversely, if $({^\circ T},\CB_\zeta)$ is a torsion pair, then $\CB_\zeta$ is closed
under direct products and $T\in \CB_\zeta$. It follows that $T$ is partial cosilting, and we apply (1) to obtain that
$({^{\circ}T},\Cogen(T))$ is a torsion pair. Then $\CB_\zeta=\Cogen(T)$, and the proof is complete.

(3)
Assume that $T$ is cosilting $R$-module. It is obvious that $\Copres(T)\subseteq\Cogen(T)$. For the reverse inclusion, let $X\in\Cogen(T)$. By \cite[Lemma 4.2.1]{Colby_Fuller:2004}, there exists a monomorphism $0 \to X/\mathrm{Rej}_{T}(X) \overset{f}\longrightarrow T^{I}$ such that $\Hom_{R}(f,T)$ is an epimorphism, where $\mathrm{Rej}_{T}(X)=\bigcap_{f\in\Hom_R(X,T)}\Ker(f)$. But $\mathrm{Rej}_{T}(X)=0$ since $X$ is a $T$-cogenerated $R$-module, hence we have an exact sequence $$0 \to X \overset{f}\longrightarrow T^{I} \overset{g}\longrightarrow Y \to 0.$$ Since $\Hom_{R}(f,T)$ is an epimorphism, the sequence $$0 \to \Ext^{1}_{R}(Y,T) \overset{\Ext^{1}_{R}(g,T)}\longrightarrow \Ext^{1}_{R}(T^{I},T)  \overset{\Ext^{1}_{R}(f,T)}\longrightarrow \Ext^{1}_{R}(X,T)$$ is exact.
But $\Ext^{1}_{R}(T^{I},T)=0$ since $\Cogen(T) \subseteq {^{\perp}T}$. It follows that $\Ext^{1}_{R}(Y,T)=0$, i.e. $Y \in {^{\perp}T}$. By Lemma \ref{closure_prop}(5), it follows that $Y\in\CB_{\zeta}$, hence $Y\in\Cogen(T)$. Therefore, $X\in\Copres(T)$.
%
\end{proof}

In order to construct partial cosilting modules, we can use the same technique as in \cite[Proposition 2.8]{Colpi_Tonolo_Trlifaj:1997} .

\begin{proposition}\label{dual-partial}
Let $S$ be a commutative ring, and let $R$ be an $S$-algebra.
If $E$ is an injective cogenerator for the category $\Modr S$, we denote by
$$(-)^d=\Hom_S(-,E):R\Mod \rightleftarrows \Modr R:\Hom_S(-,E)=(-)^d$$
the $\Hom$-contravariant functors induced by $E$.

Suppose that $$P_2\to P_1\overset{\zeta}\to P_0\to M\to 0$$ is a projective presentation of the left $R$-module $M$. The following statements are true:
\begin{enumerate}[{\rm (1)}]

\item If $X\in \CB_{\zeta^d}$ then $X^d\in\CALD_\zeta$.

\item If all projective modules $P_i$ are finitely presented and $Y\in \CALD_\zeta$ then $Y^d\in \CB_{\zeta^d}$.

\item Suppose that all projective modules $P_i$ are finitely presented. Then the left $R$-module $M$ is
partial silting with respect to $\zeta$ if and only if the dual $M^d$ is a
partial cosilting right $R$-module with respect to $\zeta^d$.
\end{enumerate}
%
%
%
\end{proposition}

\begin{proof}
Let $U$ be the image of $\zeta$ and let $\zeta=\tau_{\zeta} \circ \pi_{\zeta}$ be the canonical decomposition of $\zeta$ (i.e. $\pi_{\zeta}:P_1\to U$ and $\tau_{\zeta}:U\to P_0$). By the dual result of Lemma \ref{sigma_decomposition} we obtain that a left $R$-module $Y$ is in the class $\CALD_\zeta$ if and only if $\Ext_R^1(M,Y)=0$ and $\Hom_R(\pi_{\zeta},Y)$ is an isomorphism. Moreover, the exact sequence $$0\to M^d\to P_0^d\overset{\zeta^d}\to P_1^d$$ is an injective copresentation of $M^d$ and $\zeta^d=\pi_{\zeta}^d \circ \tau_{\zeta}^d$ is the canonical decomposition of $\zeta^d$. Therefore, a right $R$-module $X$ belongs to the class $\CB_{\zeta^d}$ if and only if $\Ext^1_R(X,M^d)=0$ and $\Hom_R(X, \pi_{\zeta}^d)$ is an isomorphism.

(1) Let $X\in \CB_{\zeta^d}$. Then we have $\Ext^1_R(X,M^d)=0$. By \cite[Lemma 2.16(b)]{Gobel_Trlifaj:2006} we observe that $\Ext^1_R(X,M^d) \cong \Tor^{R}_1(X,M)^d$, so that $\Ext^1_R(X,M^d)=0$ if and only if $\Tor^{R}_1(X,M)^d=0$, and this is equivalent to
$\Ext_R^1(M,X^d)=0$ (by the left-side version of \cite[Lemma 2.16(b)]{Gobel_Trlifaj:2006}). Now we can use \cite[Lemma 2.16(a)]{Gobel_Trlifaj:2006} and its left-side version to observe that there are natural isomorphisms of functors
$$\Hom_R(X,(-)^d)\overset{\cong}\to (X\otimes_R -)^d \overset{\cong}\to \Hom_R(-,X^d),$$
and it follows that $\Hom_{R}(\pi_{\zeta},X^d)$ is an isomorphism. Therefore, $X^d\in\CALD_\zeta$.

(2) Suppose that $Y\in\CALD_\zeta$. Then $\Ext_R^1(M,Y)=0$, hence $\Ext^{1}_R(M,Y)^d=0$, and we can use the left version of
\cite[Lemma 2.16(d)]{Gobel_Trlifaj:2006} to obtain $\Tor_1^R(Y^d,M)=0$.
Therefore, $\Tor_1^R(Y^d,M)^d=0$, and it follows by  \cite[Lemma 2.16(b)]{Gobel_Trlifaj:2006} that $\Ext_R^1(Y^d,M^d)=0$.

Moreover, $\Hom_R(\pi_{\zeta},Y)$ is an isomorphism, hence $\Hom_R(\pi_{\zeta},Y)^d$ is an isomorphism. Now we can use
the left version of \cite[Lemma 2.16(c)]{Gobel_Trlifaj:2006} to observe that
$$Y^d \otimes_R \pi_{\zeta} :Y^d \otimes_R P_1 \to Y^d\otimes_R U$$ is an
isomorphism, hence $(Y^d\otimes_R \pi_{\zeta} )^d$ is an isomorphism. By \cite[Lemma 2.16(a)]{Gobel_Trlifaj:2006} it follows that
$\Hom_R(Y^d,\pi_{\zeta}^d)$ is an isomorphism, and the proof is complete.

(3) Suppose that $M$ is partial silting with respect to $\zeta$. We will prove that $M^d$ is partial cosilting with respect to $\zeta^d$. Since $M \in \CALD_\zeta$, it follows, by (2), that  $M^d\in \CB_{\zeta^d}$.

Applying the same techniques as before, we observe that we have a natural isomorphism $$\Ext^1_R(-,M^d)\cong \Tor^{R}_1(-,M)^d.$$ But $E$ is an injective cogenerator,
so for every right $R$-module $X$ we have $$\Ext^1_R(X,M^d)=0\textrm{ if and only if }\Tor^{R}_1(X,M)=0.$$ Since the left modules
$M$ and $U$ are finitely presented, it follows by \cite[Theorem A]{strebel} that the functor $\Tor^{R}_1(-,M)$ commutes with direct products, hence the class ${^{\perp} M^d}$ is closed under direct products.

In the same way, since $E$ is an injective cogenerator we have that $\Hom_R(X,\pi_{\zeta}^d)$ is an isomorphism if and only if the homomorphism $X\otimes_R\pi_{\zeta}$ is an isomorphism, and it follows
that the class of all right $R$-modules $X$ such that $\Hom_R(X,\pi_{\zeta}^d)$ is an isomorphism is closed under direct products (we note that $P_1$ and $U$ are finitely presented, hence the tensor functor induced by them commutes with respect direct products). Therefore the class $\CB_{\zeta^d}$ is closed under direct products. Hence $M^d$ is partial cosilting with respect to $\zeta^d$.

Now, suppose that $M^d$ is partial cosilting with respect to $\zeta^{d}$. A left $R$-module $X$ is in $\CALD_\zeta$ if and only if $\Ext_R^1(M,X)=0$ and $\Hom_R(\pi_{\zeta},X)$ is an isomorphism. Since $M$, $P_1$ and $U$ are finitely presented, it follows that the functors $\Ext_R^1(M,-)$, $\Hom_R(P_{1},-)$ and $\Hom_R(U,-)$ commute with direct sums. Therefore $\CALD_\zeta$ is closed under direct sums.

In order to prove that $M\in\CALD_\zeta$, let us observe that $\Ext_R^1(M,M)=0$ if and only if $\Ext_R^1(M,M)^d=0$. Since $M$ has a projective resolution in which the first three projective modules are finitely presented, this is equivalent to $\Tor^R_1(M^d,M)=0$. Applying again \cite[Lemma 2.16(b)]{Gobel_Trlifaj:2006}, this is equivalent to $$\Ext_R^1(M^d,M^d)\cong \Tor^R_1(M^d,M)^d=0.$$ Since $M^{d} \in {^{\perp}M^{d}}$, it follows that $\Ext_R^1(M,M)=0$.

We have to prove that $\Hom_R(\pi_{\zeta},M)$ is an isomorphism. Since $E$ is an injective cogenerator, this is equivalent to $\Hom_R(\pi_{\zeta},M)^d$ is an isomorphism. Applying the Hom-Tensor relations, we obtain that $\Hom_R(\pi_{\zeta},M)^{d}$ is an isomorphism if and only if $M^d\otimes_R \pi_{\zeta}$ is an isomorphism. Again, this is equivalent to $(M^d \otimes_R \pi_{\zeta})^d$ is an isomorphism, and this is true if and only
if $\Hom_R(M^d,\pi_{\zeta}^d)$ is an isomorphism.
\end{proof}

Let us recall from \cite{Wisbauer:1991} that if
$M$ is an $R$-module then $\sigma[M]$ denotes the full subcategory of $\Modr R$ whose objects are submodules of $M$-generated modules.
Recall that $\sigma[M]$ is a Grothendieck category, and if $E$ is an injective cogenerator in $\Modr R$ then
$\mathrm{Tr}_M(E)=\sum_{f\in\Hom_R(M,E)}\Ima(f)$,
the maximal $M$-generated submodule of $E$, is an injective cogenerator for the category $\sigma[M]$.

\begin{theorem}\label{main-cosilting}
Let $T$ be an $R$-module and let $E$ be an injective cogenerator in $\Modr R$. If $0\to T\to Q_0\overset{\zeta}\to Q_1$ is an
injective copresentation for $T$ then the following are equivalent:
\begin{enumerate}[{\rm (1)}]
\item $T$ is cosilting with respect to $\zeta$;

\item $T$ has the following properties: \begin{enumerate}[{\rm (a)}]
\item $T$ is partial cosilting with respect to $\zeta$, and
\item  there exists an exact sequence
$$0\to T_1\to T_0\overset{\gamma}\to E$$
such that $T_0,T_1\in\Prod(T)$ and for every $T'\in\CB_\zeta$ the homomorphism $\Hom_{R}(T',\gamma)$ is epic.
\end{enumerate}
\end{enumerate}
\end{theorem}

\begin{proof}
(1)$\Rightarrow$(2) Suppose that $T$ is cosilting with respect to $\zeta$. Let $I$ be the torsion part of $R$ with respect to the torsion pair
$({^{\circ}T},\Cogen(T))$. Then for every $X\in \Cogen(T)$ the canonical homomorphism $\Hom_R(R/I,X)\to \Hom_R(R,X)$ is an isomorphism,
hence $X$ is $R/I$-generated. In particular $\Cogen(T)\subseteq \sigma[R/I]$.

Let $\alpha:R/I^{(\Lambda)}\to \mathrm{Tr}_{R/I}(E)$ be an epimorphism.
Since $R/I\in \Cogen(T)$, it follows that $R/I^{(\Lambda)}\in\Cogen(T)$, hence we have a monomorphism
$0\to R/I^{(\Lambda)}\to T^{\Omega}$. We can view this monomorphism in the category $\sigma[R/I]$, and we obtain that $\alpha$ can be
extended to an epimorphism $\beta:T^{\Omega}\to \mathrm{Tr}_{R/I}(E)$, by the injectivity of $\mathrm{Tr}_{R/I}(E)$ in the category $\sigma[R/I]$. Now the kernel $K=\Ker(\beta)$ is $T$-cogenerated, hence, taking into account Corollary \ref{cogen=copres}, it is $T$-copresented. It follows that we have an exact sequence $$0\to K\to T^{Z}\to L\to 0$$ such that $L\in\Cogen(T)$.

Using all these data we construct the pushout diagram
\[\xymatrix{
 & 0 \ar[d] & 0 \ar[d] & & \\
0\ar[r] & K \ar[d]\ar[r] & T^{\Omega} \ar[d]\ar[r] & \mathrm{Tr}_{R/I}(E)\ar@{=}[d]\ar[r] &0 \\
0\ar[r] & T^{Z} \ar[d]\ar[r] & X \ar[d]\ar[r] & \mathrm{Tr}_{R/I}(E)\ar[r] & 0 \\
 & L \ar[d]\ar@{=}[r] & L \ar[d] & & \\
 & 0  & 0  & & .\\
}\]
Since $\Cogen(T)$ is a torsion-free class, it  is closed under extensions, and it follows, by the second vertical exact sequence, that $X\in\Cogen(T)$. Therefore, the exact sequence
$$0\to T^{Z}\to X\to \mathrm{Tr}_{R/I}(E)\to 0$$
is in $\sigma[R/I]$.

Let $V\in\Cogen(T)$. If we apply the covariant extension functor $\Ext_{\sigma[R/I]}(V,-)$ on the above exact sequence, we obtain the following exact sequence of abelian groups $$\Ext_{\sigma[R/I]}(V,T^Z)\to \Ext_{\sigma[R/I]}(V,X)\to \Ext_{\sigma[R/I]}(V,\mathrm{Tr}_{R/I}(E)).$$
In this exact sequence the first group is zero since $$\Ext_{\sigma[R/I]}(V,T^Z)\leq \Ext^1_{R}(V,T^Z)=0,$$ and the third group
is also trivial since $\mathrm{Tr}_{R/I}(E)$ is injective in $\sigma[R/I]$.

Therefore, every short exact sequence $$0\to X\to U\to V\to 0$$
with $U\in\sigma[R/I]$ and $V\in\Cogen(T)$ splits.

By Corollary \ref{cogen=copres}, there
exists a short exact sequence $$0\to X\to T^{\Gamma}\to W\to 0$$ with $W\in \Cogen(T).$ Since this is an exact sequence in $\sigma[R/I]$ and
$\Cogen(T)\subseteq\sigma[R/I]$, it splits,
hence $X\in\Prod(T)$.
Therefore, in the short exact sequence  $$0\to T^Z\to X\to \mathrm{Tr}_{R/I}(E)\to 0$$ the first two terms are in $\Prod(T)$ and
$\mathrm{Tr}_{R/I}(E)$ is a submodule of $E$, so it can be viewed as an exact sequence
\begin{equation*}\tag{$\sharp$} 0\to T^Z\to X\overset{\gamma}\to E\end{equation*} as in (b).

{Moreover, if $U$ belongs to $\CB_\zeta=\Cogen(T)$ and $\alpha:U\to E$ is a homomorphism then $\Ima(\alpha) \subseteq \mathrm{Tr}_{R/I}(E)$, since $U$ is an $R/I$-generated module  by what has been shown so far. Then $\Hom_{R}(U,E) = \Hom_{R}(U,\mathrm{Tr}_{R/I}(E))$, and we obtain that the natural homomorphism $$\Hom_R(U,\gamma):\Hom_R(U,X)\to \Hom_R(U,E)$$ is epic since $\Ext^1_R(U,T^Z)=0$.
Therefore, the exact sequence $(\sharp)$ satisfies all conditions stated in (b).
}

(2)$\Rightarrow$(1) It is enough to prove that $\CB_\zeta\subseteq \Cogen(T)$. Let $X\in \CB_\zeta$. Let $E$ be an injective cogenerator for $\Modr R$. Then, for every non-zero element $x\in X$, there exists a homomorphism $\alpha:X\to E$ such that $\alpha(x)\neq 0$. By (b) we can lift
$\alpha$ to a homomorphism $\overline{\alpha}:X\to T_0$. From $T_0\in\Prod(T)$, it follows that, for every $0\neq x\in X$, there exists a homomorphism
$\alpha':X\to T$ such that $\alpha'(x)\neq 0$. It follows that $X$ is $T$-cogenerated, and the proof is complete.
\end{proof}




We obtain the cosilting version of \cite[Proposition 5.2.7]{Colby_Fuller:2004}.

\begin{corollary}\label{dual}
Suppose that we are in the same hypotheses as in Proposition {\rm \ref{dual-partial}(3)}.
\begin{enumerate}[{\rm (1)}]

\item If $M$ is a silting module with respect to $\zeta$ then $M^d$ is a cosilting module with respect to $\zeta^d$.

\item
Suppose that $R$ is an Artin algebra, $S$ is the center of $R$ and $(-)^d$ is the standard duality between finitely presented left and
right modules induced by the injective
envelope of $S/J(S)$.
Then $M$ is a silting module with respect to $\zeta$ if and only if
$M^d$ is a cosilting module with respect to $\zeta^d$.
\end{enumerate}
\end{corollary}

\begin{proof}
(1) Suppose that $M$ is silting with respect to $\zeta$. By \cite[Proposition 3.11]{Angeleri_Marks_Vitoria:2015} there exists an exact sequence
$$R\overset{\alpha}\to M_0\to M_1\to 0$$ such that $M_0, M_1\in\Add(M)$, and for every left $R$-module $Y\in \CALD_\zeta$ the homomorphism
$\Hom_{R}(\alpha,Y)$ is an epimorphism. Applying the functor $(-)^d$ we obtain an exact sequence
$$0\to M_1^d\to M_0^d\overset{\alpha^d}\to R^d$$ such that $M_0^d,M_1^d\in\Prod(M^d)$ and $R^d$ is an injective cogenerator for $\Modr R$. By
Theorem \ref{main-cosilting} and Proposition \ref{dual-partial}(3), it is enough to prove that for every $X\in\CB_{\zeta^d}$ the
homomorphism $\Hom_R(X,\alpha^d)$ is epic.

Let $X\in \CB_{\zeta^d}$. By Proposition \ref{dual-partial}(1) we know that $X^d\in\CALD_\zeta$, so the homomorphism
$\Hom_R(\alpha,X^d)$ is epic. Moreover, for every $Y\in R\Mod$ we have the natural isomorphisms
$$\Hom_R(X,Y^d)\cong (X\otimes_R Y)^d\cong \Hom_{R}(Y,X^d),$$ and if we take $Y=M_0$, respectively $Y=R$, we obtain that $\Hom_{R}(X,\alpha^d)$ is
an epimorphism.

{(2) Suppose that $T=M^d$ is a cosilting right $R$-module with respect to $\zeta^d$. Then $T$ is partial cosilting with respect to $\zeta^d$.
We observe that in the proof of Theorem \ref{main-cosilting} we can choose $E=R^d$ as an injective cogenerator. Since $E$ is finitely generated we can suppose that the set $\Lambda$ is finite. Therefore $(R/I)^{(\Lambda)}$ is finitely generated. Moreover, using \cite[Corollary 1.3.3]{Colby_Fuller:2004}, it follows that $T^{\Omega}$ is a direct summand of a direct sum $T^{(\Gamma)}$ of copies of $T$. Since $(R/I)^{(\Lambda)}$ is finitely generated, it follows that we can embed it in a finite direct sum of copies of $T$. Therefore we can suppose that $\Omega$ is finite. Therefore $K$ is finitely generated. Let $0\to K\to T^Z\to L\to 0$ be an exact sequence such that $L\in \Cogen(T)$. Since $T^Z$ is a direct summand of a direct sum $T^{(Z')}$ it follows that we can embed $K$ in a short exact sequence 
$0\to K\overset{f}\to T^{(Z')}\to L'\to 0$ such that $L'\in \Cogen(T)$. But $K$ is finitely generated, hence there exists a finite subset $Z''$ of $Z'$ such that $\alpha(K)\subseteq T^{(Z'')}$, and it is not hard to see that $T^{(Z'')}/f(K)\in\Cogen(T)$ since it can be embedded in $L'$. Therefore, in the proof of Theorem \ref{main-cosilting} we can suppose that $Z$ is finite. 
Then the module $X$ used in the proof of this theorem is finitely generated, and we use one more time \cite[Corollary 1.3.3]{Colby_Fuller:2004} to conclude that $X\in \mathrm{add}(T)$, where $\mathrm{add}(T)$ denotes the class of all direct summands of finite direct sums of copies of $T=M^d$.
}

Therefore, we can construct an exact sequence
$$0\to T_1\to T_0\overset{\gamma}\to E$$ such that $E=R^d$ and $T_0,T_1\in\mathrm{add}(M^d)$. Since the functors $(-)^d$ give a duality
between finitely presented left and right $R$-modules we can write this exact sequence as
$$0\to T_1\to T_0^{dd}\overset{\gamma^{dd}}\to E^{dd},$$
and we observe that for every $Y\in \CB_{\zeta^d}$ we have that
$$\Hom_R(Y,\gamma^{dd}):\Hom_{R}(Y,T_0^{dd})\to \Hom_{R}(Y,E^{dd})$$
is an epimorphism. This implies that $$(Y\otimes_R \gamma^d)^d:(Y\otimes_R T_0^d)^d\to (Y\otimes_R E^d)^d$$ is an epimorphism, so
$Y\otimes_R\gamma^d$ is a monomorphism.

Let $X\in\CALD_\zeta$.  By Proposition \ref{dual-partial}(2) we obtain that $X^d\in \CB_{\zeta^d}$, so
$$(X^d\otimes_R \gamma^d):X^d\otimes_R E^d \to X^d\otimes_R T_0^d$$ is a monomorphism. This implies that
$$\Hom_R(\gamma^d,X)^d:\Hom_R(E^d,X)^d\to \Hom_{R}(T_0^d,X)^d$$ is a monomorphism, hence $\Hom_{R}(\gamma^d,X)$ is an epimorphism.

Therefore, since $R=E^d$, we have an exact sequence $$R\overset{\gamma^d}\to T_0^d\to T_1^d \to 0$$ such that
$T_0^d$, $T_1^d\in \mathrm{add}(M)$, and for every $X\in \CALD_{\zeta}$ the homomorphism $\Hom_{R}(\gamma^d,X)$ is epic.
By \cite[Proposition 3.11]{Angeleri_Marks_Vitoria:2015} we obtain that $M$ is silting with respect to $\zeta$, and the proof is complete.
\end{proof}

\begin{remark}(added in proof) It was proved by Angeleri and Hrbek in \cite[Proposition 3.4]{An_Hr:2017} that the statement (1) in Corollary \ref{dual} is valid for all silting modules. The main ingredient in this general proof is the fact that the classes of all modules which are (co)generated by (co)silting modules are definable (we refer to Corollary \ref{definable} for the cosilting case).
\end{remark}

For reader's convenience we will present an example which shows that the converse proved in Corollary \ref{dual}(2) is not valid for the general case.

\begin{example}\label{ex1-d}
 Let $\PP$ be the set of all primes.
If $R=S=\Z$, we consider the contravariant functor $(-)^d=\Hom(-,\prod_{p\in\PP}\Z(p^\infty)):\Modr \Z\to\Modr \Z$ induced by the direct product $\prod_{p\in\PP}\Z(p^\infty)$ of all injective envelopes for all simple $\Z$-modules. If $M=\langle \frac{1}{p}\mid p\in \PP\rangle$ then for every prime $p$ we consider a short exact sequence $0\to L\to M\to \oplus_{q\neq p}\Z(q)\to 0$ such that $L\cong \Z$, and it follows that $\Hom(M,\Z(p^\infty))\cong \Hom(L,\Z(p^\infty))\cong
\Z(p^\infty)$. Therefore, $M^d\cong \prod_{p\in\PP}\Z(p^\infty)$. If $0\to F_{-1}\overset{\zeta}\to F_0\to M\to 0$ is a projective (free) resolution for $M$ then we obtain a split exact sequence of injective abelian groups $0\to M^d\to F_0^d\overset{\zeta^d}\to F_{-1}^d\to 0$. Since $M^d$ is an injective cogenerator for $\Modr \Z$, it follows that $M^d$ is cosilting with respect to $\zeta^d$. But $M$ is not (partial) silting since $\Ext^1_\Z(M,M)\neq 0$.
\end{example}

We mention that if we start with a cosilting module $N$ then its dual $N^d$ is not necessarily a (partial) silting module. In order to see this we will use the examples constructed in \cite[Section 3]{Ang-proc}.

\begin{example}
We will use the same setting as in Example \ref{ex1-d}. Then $N=\BBQ\oplus \BBQ/\BBZ$ is a cotilting $\BBZ$-module. The dual module $N^d=\Hom(N,\prod_{p\in\PP}\Z(p^\infty))$ contains a direct summand isomorphic to $\BBQ$ and a direct summand isomorphic to the group $\widehat{\BBZ}_p$ of all $p$-adic integers. If we suppose that $N^d$ is partial silting then we obtain, using \cite[Remark 3.8(1)]{Angeleri_Marks_Vitoria:2015},  that $\Ext^1_\Z(\BBQ,\widehat{\BBZ}_p^{(\omega)})=0$, hence $\widehat{\BBZ}_p^{(\omega)}$ is a cotorsion group. But this is not true, by  \cite[Proposition 1.7]{BaSt}. Therefore, $N^d$ is not (partial) silting. By a similar argument it follows that the example constructed in \cite[Section 3.4]{Ang-proc} can be used to see that over tame hereditary artin algebras of infinite representation type there exists a cotilting module $W$ such that its dual $W^d$ is not (partial) silting.
\end{example}

We thank Lidia Angeleri-H\"ugel for indicating us the following result: the dual of \cite[Theorem 3.12]{Angeleri_Marks_Vitoria:2015} is also valid.

\begin{theorem}\label{pcs-ds}
Let $T$ be a partial cosilting $R$-module with respect to an injective copresentation $0\to T\to Q_0\overset{\zeta}\to Q_1$.
Then there exists an $R$-module $M$ and an injective copresentation  $0\to T\oplus M\to Q'_0\overset{\zeta'}\to Q'_1$ such that $T\oplus M$
is cosilting with respect to $\zeta'$ and $\CB_\zeta=\CB_{\zeta'}$.
\end{theorem}

\begin{proof}
Let $E$ be an injective cogenerator for $\Modr R$. If $\Lambda=\Hom_R(E,Q_1)$, and $\psi:E\to Q_1^\Lambda$ is the canonical homomorphism,
then we denote by $M$ the pullback of the diagram $Q_0^\Lambda\overset{\zeta^\Lambda}\rightarrow Q_1^\Lambda\overset{\psi}\leftarrow E$.
Then $M$ is the kernel of the map $\gamma:Q_0^\Lambda \oplus E\to Q_1^\Lambda$ induced by $\zeta^\Lambda$ and $\psi$. We consider the
exact sequence $$0\to M\oplus T\to Q_0^\Lambda \oplus E\oplus Q_0\overset{\gamma\oplus\zeta}\rightarrow Q_1^\Lambda\oplus Q_1,$$ and we will
prove that $M\oplus T$ is cosilting with respect to $\gamma\oplus\zeta$.

Note that $\CB_{\gamma\oplus\zeta}=\CB_\gamma\cap\CB_\zeta$. Let  $X$ be a module in $\CB_\zeta$. Then $X\in \CB_{\zeta^\Lambda}$. Since
every homomorphism $\alpha:X\to Q_1^\Lambda$ factorizes as $\alpha=\zeta^\Lambda\beta$, with $\beta:X\to Q_0^\Lambda$, it follows that
$\alpha=\gamma \delta$, where $\delta:X\to E\oplus Q_0^\Lambda$ is the homomorphism induced by $\beta$ and $0$. Therefore,
$\CB_\zeta\subseteq \CB_\gamma$, hence $\CB_{\gamma\oplus \zeta}=\CB_\zeta$, hence $\CB_{\gamma\oplus \zeta}$ is closed under
direct products.

Now, we will prove that $M\in \CB_\zeta$. Let $\alpha:M\to Q_1$ be a homomorphism.
We embed $M$ in the solid rectangle of the following commutative diagram
\[\xymatrix{0 \ar[r] & T^\Lambda\ar[r]^{\upsilon}\ar@{=}[d]\ar@{..>}[ddr]^(.3){\xi} &
M\ar[r]^{\epsilon}\ar[d]^{\mu}\ar@{..>}[ddr]^(0.3){\alpha}
\ar@/_1pc/@{..>}[dd]_(.3){\rho}
& E\ar[d]^{\psi}\ar@/^2pc/@{.>}[dd]^{\lambda} \\
   0 \ar[r] & T^\Lambda\ar[r]\ar@{-->}[d]^{\pi_\lambda} & Q_0^\Lambda\ar[r]^(.7){\zeta^{\Lambda}}\ar@{-->}[d]^{\pi_\lambda} &
	Q_1^\Lambda \ar@{-->}[d]^{\pi_\lambda} \\
	0 \ar[r] & T\ar[r] & Q_0\ar[r]^{\zeta} & Q_1 ,
	}\]
and we observe that $\alpha\upsilon:T^\Lambda\to Q_1$ factorizes through $\zeta$ (since $T^\Lambda\in\CB_\zeta$),
hence $\alpha\upsilon=\zeta\xi$ with $\xi:T^\Lambda\to Q_0$.
Since $Q_0$ is injective there exists a homomorphism $\rho:M\to Q_0$ such that $\xi=\rho\upsilon$. Note that $\zeta\rho-\alpha$ factorize
through $\epsilon$, hence there exists $\lambda:E\to Q_1$ such that $\lambda \epsilon=\zeta\rho-\alpha$. In the above diagram
$\pi_\lambda$ denote the canonical projections onto the $\lambda$-th corresponding components. Since $\pi_\lambda\psi=\lambda$ we obtain
$\zeta\rho-\alpha=\lambda\epsilon=\pi_\lambda\psi\epsilon=\zeta\pi_\lambda\mu$, hence $\alpha$ factorizes through $\zeta$. Then $M\in
\CB_\zeta=\CB_{\gamma\oplus \zeta}$.
It follows that $M\oplus T$ is partial cosilting with respect to $\gamma\oplus\zeta$.

By Theorem \ref{main-cosilting}, in order to complete the proof, it is enough to prove that every homomorphism $\beta:M'\to E$,
with $M'\in\CB_\zeta$, factorize through $\epsilon$. This is true since $\psi\beta$ factorizes through $\zeta^\Lambda$, and $M$ is constructed as a pullback of $\psi$ and $\zeta^\Lambda$.
\end{proof}

The following Corollary is already known, and it can be deduced from \cite[p.93]{An_Co:2002} and the dual of Bongarz Lemma,
\cite[Proposition 6.43]{Gobel_Trlifaj:2006}.

\begin{corollary}
Every partial cotilting module is a direct summand of a cotilting module.
\end{corollary}
\begin{proof}
If in the proof of Theorem \ref{pcs-ds} we assume that $\zeta$ is an epimorphism, then we obtain that $\gamma\oplus\zeta$ is an epimorphism,
so $M\oplus T$ is cotilting by Example \ref{si.vs.ti}.
\end{proof}

\section{Cosilting Modules are Pure Injective}

We recall that a short exact sequence in $\Modr R$ is said to be {\sl pure-exact} if the covariant functor $\Hom_{R}(F,-)$ preserves its exactness for every finitely presented module $F$. An $R$-module $U$ is {\sl pure-injective} if the contravariant functor $\Hom_{R}(-,U)$ preserves the exactness of every pure-exact sequence.

In order to prove the main result of this section, i.e. all cosilting modules are pure-injective, we need the following characterization of pure-injective modules.

\begin{proposition}\label{pure-injective} \cite[Theorem 7.1]{Jensen_Lenzing:1989} An $R$-module $U$ is pure-injective if and only if the contravariant functor $\Hom_{R}(-,U)$ preserves the exactness of the canonical short exact sequence $$0 \to U^{(\lambda)} \longrightarrow U^{\lambda} \longrightarrow U^{\lambda}/U^{(\lambda)} \to 0$$ for every cardinal $\lambda$.
\end{proposition}

We also need some useful lemmas related to cardinals. Let $I$ be a set and let $\CI$ be a family of subsets of $I$. We say that $\CI$ is {\sl almost disjoint} if the intersection of any two distinct elements of $\CI$ is finite.

\begin{lemma}\label{almost_disjoint} \cite[Lemma 2.3]{Bazzoni:2003} Let $\lambda$ be an infinite cardinal. Then there is a family of $\lambda^{\aleph_{0}}$ countable almost disjoint subsets of $\lambda$.
\end{lemma}

\begin{lemma}\label{cardinal_equality} \cite[Lemma 3.1]{Griffith:1970} For any cardinal $\gamma$, there is an infinite cardinal $\lambda\geq\gamma$ such that $\lambda^{\aleph_{0}}=2^{\lambda}$. Consequently, for every cardinal $\gamma$, there is an infinite cardinal $\lambda$ such that $\gamma \leq \lambda < \lambda^{\aleph_{0}}$.
\end{lemma}

We have the following construction given by Bazzoni in the proof of \cite[Proposition 2.5]{Bazzoni:2003}. We present the
details for reader's convenience.

\begin{proposition}\label{Bazzoni} Let $T$ be an $R$-module and let $\lambda$ be an infinite cardinal. Then there is a submodule $T^{(\lambda)} \leq V \leq T^{\lambda}$ such that $V/T^{(\lambda)} \cong X^{(\lambda^{\aleph_{0}})}$, where $X = T^{\aleph_{0}}/T^{(\aleph_{0})}$.
\end{proposition}
\begin{proof} By Lemma \ref{almost_disjoint}, there is a family $\CI = \{I_{\alpha} \} _{\alpha \leq \lambda^{\aleph_{0}}}$ of $\lambda^{\aleph_{0}}$ countable almost disjoint subsets of
$\lambda$. We view $T^{I_{\alpha}}$ as embedded in $T^{\lambda}$ and we consider $p_{\alpha} : T^{I_{\alpha}} \to T^{\lambda}/T^{(\lambda)}$ be the restriction to $T^{I_{\alpha}}$ of the canonical projection $p : T^{\lambda} \to T^{\lambda}/T^{(\lambda)}$, for all $\alpha \leq \lambda^{\aleph_{0}}$. We note that $\Ker(p_{\alpha}) = T^{(I_{\alpha})}$ and $\Ima(p_{\alpha}) \cong X$, for all $\alpha \leq \lambda^{\aleph_{0}}$. Moreover, the sum $\sum_{\alpha \leq \lambda^{\aleph_{0}}} \Ima(p_{\alpha})$ in $T^{\lambda}/T^{(\lambda)}$ is actually a direct sum.

Since all $\Ima(p_{\alpha})$ are submodules of $T^{\lambda}/T^{(\lambda)}$, it follows that $\bigoplus_{\alpha \leq \lambda^{\aleph_{0}}} \Ima(p_{\alpha}) \cong X^{(\lambda^{\aleph_{0}})}$ is also a submodule of $T^{\lambda}/T^{(\lambda)}$, hence there is a submodule $T^{(\lambda)} \leq V \leq T^{\lambda}$ such that $V/T^{(\lambda)} \cong X^{(\lambda^{\aleph_{0}})}$.
\end{proof}

\begin{proposition}\label{countable_exactness} Let $T$ be a partial cosilting $R$-module
with respect to the injective copresentation $\zeta : Q_{0} \to Q_{1}$. Then $T^{\aleph_{0}}/T^{(\aleph_{0})} \in \CB_\zeta$.
\end{proposition}
\begin{proof} Suppose that $T\neq 0$ is partial cosilting with respect to the injective copresentation $$0 \to T \overset{f}\longrightarrow Q_{0} \overset{\zeta}\longrightarrow Q_{1}.$$ By definition, $T$ belongs to $\CB_{\zeta}$, hence
for every cardinal $\gamma$ we have $T^{(\gamma)}\in \CB_\zeta$ (by Lemma \ref{closure_prop}) and $T^{\gamma}\in\CB_{\zeta}$. By Lemma \ref{inclusions}, we obtain $$\Ext^{1}_{R}(T^{\gamma},T) = 0 = \Ext^{1}_{R}(T^{(\gamma)},T).$$

By Lemma \ref{last_term} it is enough to prove that $T^{\aleph_{0}}/T^{(\aleph_{0})} \in {^{\perp}T}$.
We denote by $X$ the $R$-module $T^{\aleph_{0}}/T^{(\aleph_{0})}$.

Using Lemma
\ref{cardinal_equality} we can fix an infinite cardinal $\lambda \geq  \left|
\Hom_{R}(T,T) \right|$ such that $\lambda^{\aleph_{0}} = 2^{\lambda}$.
We note that from $\Hom_{R}(T,T) \neq 0$ we have $\lambda \geq \left| \Hom_{R}(T,T) \right| \geq 2$, and it follows that $\left| \Hom_{R}(T,T) \right|^{\lambda} = 2^{\lambda}$, since
$2^\lambda\leq \left| \Hom_{R}(T,T) \right|^{\lambda}\leq \lambda^\lambda = \lambda^{\aleph_0\lambda} = 2^{\lambda\lambda} = 2^\lambda$.

Applying $\Hom_{R}(-,T)$ to the canonical short exact sequence $$0 \to
T^{(\lambda)} \overset{i}\longrightarrow T^{\lambda}
\overset{p}\longrightarrow T^{\lambda}/T^{(\lambda)} \to 0,$$we obtain the inequality $$\left| \Hom_{R}(T,T) \right|^{\lambda} \geq \left| \Ext^{1}_{R}(T^{\lambda}/T^{(\lambda)},T) \right|.$$

By Proposition \ref{Bazzoni}, there is a submodule $T^{(\lambda)} \leq V \leq T^{\lambda}$ such that $V/T^{(\lambda)} \cong X^{(\lambda^{\aleph_{0}})}$.
Since $\CB_{\zeta}$ is closed under submodules (cf. Lemma \ref{closure_prop}), we observe that $V$ belongs to $\CB_{\zeta}$ hence $\Ext^{1}_{R}(V,T) = 0$. Applying the $\Hom_{R}(-,T)$ functor to the following pushout commutative diagram
$$\begin{CD}
@. 0 @. 0\\
@. @VVV @VVV\\
@. T^{(\lambda)} @= T^{(\lambda)}\\
@. @VrVV @VViV\\
0 @>>> V @>u>> T^{\lambda} @>v>> T^{\lambda}/V @>>> 0\\
@. @VsVV @VVpV @|\\
0 @>>> V/T^{(\lambda)} @>j>> T^{\lambda}/T^{(\lambda)} @>q>> T^{\lambda}/V @>>> 0\\
@. @VVV @VVV\\
@. 0 @. 0\\
\end{CD}$$
we obtain the commutative diagram
$$\begin{CD}
@. 0 @>>>\Ext^{2}_{R}(T^{\lambda}/V,T)\\
@. @AAA @|\\
\Ext^{1}_{R}(T^{\lambda}/T^{(\lambda)},T) @>>> \Ext^{1}_{R}(V/T^{(\lambda)},T) @>>> \Ext^{2}_{R}(T^{\lambda}/V,T).\\
\end{CD}$$
It follows that the natural homomorphism $$\Ext^{1}_{R}(T^{\lambda}/T^{(\lambda)},T) \longrightarrow \Ext^{1}_{R}(V/T^{(\lambda)},T)$$ is an epimorphism,
hence we have the inequality between cardinals $$\left| \Ext^{1}_{R}(T^{\lambda}/T^{(\lambda)},T) \right| \geq \left| \Ext^{1}_{R}(V/T^{(\lambda)},T) \right| = \left| \Ext^{1}_{R}(X,T) \right| ^{\lambda^{\aleph_{0}}}.$$

Now, if we assume that $\Ext^{1}_{R}(X,T) \neq 0$, then $\left| \Ext^{1}_{R}(X,T) \right| \geq 2$, and we obtain
$$2^{\lambda} = \left| \Hom_{R}(T,T) \right|^{\lambda} \geq \left| \Ext^{1}_{R}(X,T) \right|^{\lambda^{\aleph_{0}}} \geq 2^{\lambda^{\aleph_{0}}}=2^{2^\lambda}, $$ which is a contradiction. It follows that $\Ext^{1}_{R}(X,T) = 0$, hence $T^{\aleph_{0}}/T^{(\aleph_{0})} \in {^{\perp}T}$, and the proof
is complete.
%
\end{proof}

We have the following result which holds in general.
\begin{proposition}\label{co_gen_eral} \cite[Lemma 2.7]{Bazzoni:2003} Let $X$ and $U$ be two $R$-modules. If $X^{\aleph_{0}}/X^{(\aleph_{0})} \in \Cogen(U)$ then $X^{\lambda}/X^{(\lambda)} \in \Cogen(U)$, for every cardinal $\lambda$.
\end{proposition}


\begin{theorem}\label{cosi-pi} Let $T$ be an $R$-module. If $T$ is partial cosilting then $T$ is pure-injective.
\end{theorem}

\begin{proof} By Theorem \ref{pcs-ds} it is enough to prove that every cosilting module is pure-injective.
In order to obtain this we will prove that for every cardinal $\lambda$ then functor $\Hom_{R}(-,T)$ preserves the exactness of the canonical short exact sequence $$0 \to T^{(\lambda)} \longrightarrow T^{\lambda} \longrightarrow T^{\lambda}/T^{(\lambda)} \to 0.$$

Suppose that $T$ is cosilting with respect to $\zeta$. By Proposition \ref{countable_exactness}, we have $T^{\aleph_{0}}/T^{(\aleph_{0})}\in\CB_{\zeta}$. Moreover, $T^{\aleph_{0}}/T^{(\aleph_{0})}\in\CB_\zeta=\Cogen(T)$, and applying Proposition \ref{co_gen_eral} we obtain
$T^{\lambda}/T^{(\lambda)} \in \Cogen(T)$ for every cardinal $\lambda$. It follows that $T^{\lambda}/T^{(\lambda)} \in {^{\perp}T}$ for every cardinal $\lambda$, and it follows that $T$ is pure-injective.
\end{proof}

We recall that a full subcategory (or a subclass) of $\Modr R$ is a {\sl definable subcategory} if it is closed in $\Modr R$ under direct products, direct limits and pure submodules.

\begin{corollary}\label{definable}
If $T$ is a cosilting $R$-module with respect to $\zeta$, then the class $\Cogen(T)=\CB_{\zeta}$ is definable.
\end{corollary}

\begin{proof}
It is enough to prove that $\CB_{\zeta}$ is closed under direct limits. If $B=\underrightarrow{\lim}_{i\in I}B_i$ such that
all $B_i$ are modules from $\CB_{\zeta}$, then there exists a pure exact sequence
$$0\to L\to \bigoplus_{i\in I} B_i\to B\to 0$$
by \cite[33.9(2)]{Wisbauer:1991}.  Since $T$ is pure injective, it follows that $B\in {^{\perp}T}$.
By Lemma \ref{last_term}(5) we obtain $B\in \CB_\zeta$, and the proof is complete.
\end{proof}

\section*{Acknowledgment} We would like to thank the referee for his/her suggestions and comments to improve this article.

\end{document}